\documentclass[11pt,a4paper]{amsart}
\usepackage{amssymb,amscd,amsmath,young, stmaryrd,tikz}
\usepackage{mathrsfs, mathdots} \usepackage[mathcal]{eucal}
\usepackage{float}  \usepackage{soul}
\usepackage{hyperref}
\usepackage{amsrefs}
\usepackage{times}
\usepackage{wasysym,stmaryrd,enumerate}
\usepackage{amssymb,amsfonts,amsrefs,amsthm,mathrsfs,amsmath,amscd,version,graphicx,bbm, xspace}
\usepackage{tikz-cd}
\usepackage{tikz}
\usepackage{amsfonts}
\usepackage{graphicx}
\usepackage{pb-diagram}
\usepackage{epstopdf}
\usepackage{fancyhdr}
\usepackage[all]{xy}

\newtheorem{theo}{Theorem}[section]
\newtheorem{theorem}[theo]{Theorem}
\newtheorem{corollary}[theo]{Corollary}

\newtheorem{lemma}[theo]{Lemma}
\newtheorem{propo}[theo]{Proposition}

\newtheorem{theor}{Theorem}

\newtheorem{cor}[theor]{Corollary}

\theoremstyle{definition}
\newtheorem*{acknow}{Acknowledgements}

 \newtheorem{remark}[theo]{Remark}

\newtheorem{question}[theo]{Question}	
	
 \newtheorem{definition}[theo]{Definition}

\numberwithin{equation}{section}

  \author{Angela Carnevale} \address{School of Mathematics, Statistics and Applied Mathematics\\National University of Ireland, Galway}
  \email{angela.carnevale@nuigalway.ie}
  \author{Matteo Cavaleri} \address{Universit\`a degli Studi Niccol\`o Cusano - Via Don Carlo Gnocchi, 3 \\ 00166 Roma, Italia} 
  \email{matte.cavaleri@gmail.com}

\begin{document}
  \title{Partial Word and Equality problems and Banach densities}
    \date{\today}
\begin{abstract}
We investigate   partial Equality and Word Problems for finitely generated groups. After introducing Upper Banach (UB) density on free groups, we prove that solvability of the Equality Problem on squares of  UB-generic sets implies solvability of the whole Word Problem. In particular, we prove that solvability of generic EP implies WP. 
We then exploit another definition of generic EP, which turns out to be equivalent to  generic WP. We characterize in different ways the class of groups with unsolvable UB-generic WP, proving that it contains that of infinite algorithmically finite groups, and it is contained in that of groups with unsolvable generic WP.
 \end{abstract}
\maketitle
\section{introduction}
One of the most striking results of the last century in group theory is certainly the existence of finitely presented groups with \emph{unsolvable Word Problem (WP)}, independently proved in  \cite{Boone} and \cite{Novikov}.
From a practical point of view, in computability and complexity theory it is often interesting to know the behavior of an algorithm on  \emph{almost all}  inputs. 
A formalization of this approach, especially for the classical decision problems for groups, was given in \cite{generic02}: the \emph{generic version of a problem} is solvable if it is solvable on a \emph{generic subset} of the input. A similar idea was already developed in group theory,  essentially  by Gromov~\cite{ gromov}, and was given a rigorous formulation by  Arzhantseva and Olshanskii~\cite{AO}.
With this new generic approach, most of the known  examples of unsolvable decision problems on groups turned out to be generically solvable, possibly even in linear time; see, for instance,~\cite{BMR09,BMR07,average,generic02,KSS,KS05,MR}. This could be an issue, for example, for applications in group-based cryptography \cite{group-based}.  
Remaining in context of the Word Problem, to the best of our knowledge, it is still unknown if there exists a finitely presented group with unsolvable generic WP. Various partial results have been obtained in this direction.
 In \cite{Myas/11}, computably presented, infinite, \emph{algorithmically finite} groups (so-called \emph{Dehn monsters}) were found. An algorithmically finite group is a group for which the Equality Problem is ``extremely undecidable'': it is impossible to computably enumerate infinitely many pairwise distinct elements. It turns out that, with a suitable definition of the \emph{partial Equality Problem (EP)}, infinite algorithmically finite groups can have solvable EP only on \emph{negligible} sets. Moreover, the work \cite{Myas/11} raised the question about the existence of  finitely presented Dehn monsters, or at least of finitely presented groups whose EP is solvable only on non-generic sets. The first question is still open, other developments can be found in \cites{khou/14,Kly/15}. For the latter question,  the second author exhibited  finitely presented groups with unsolvable \emph{generic Equality Problem}  \cite{Cavaleri/18}.
 
 The following is the first main result of this article, settled in the context of finitely generated groups, and it gives a more complete answer to the question raised in \cite[Problem 1.5, b]{Myas/11}. It is proved at the end of Section~\ref{sec:ubgwp}. We refer the reader to Section \ref{notation} for the relevant definitions.
 \begin{theor}\label{thA}
   Let $\Gamma$ be  a finitely generated group. 
If there exists a finite set of generators $X$ of $\Gamma$ such that the Equality Problem is solvable on a set $S\times S$, where
$S\subset \mathbb F_X$ is generic (that is, the generic EP of $\Gamma$
is solvable in the sense of~\cite{Myas/11}), then $\Gamma$ has solvable Word Problem.
\end{theor}

In particular, no further assumptions are made on the group: this result holds for groups which are not necessarily amenable, computably presented (as it instead was in \cite{Cavaleri/18}). Note that, as a byproduct, this theorem shows for the first time that
the generic EP  in the sense of \cite{Myas/11} does not depend on the choice of the finite set of generators.
There is a simple idea behind this claim: up to left (or right) translations, a generic set contains all information about the whole Word Problem. 
We formalize this concept by introducing and studying \emph{Banach densities} on free groups, densities that are, in a precise way, invariant under the action of an infinite sequence of translations. While the name of our densities refers to their  classical analogues on~$\mathbb Z$, the ideas leading to their definition and applications were partly inspired from the densities defined and studied by Solecki for any discrete group in \cite{soleckihaar}. 

They turn out to have other good invariance properties. For instance, the set of trivial words is negligible in a strong sense (cf.\ Theorem \ref{osse}), which is a fundamental feature for investigations in genericity problems \cite{generic02, gilman/10}. 
We actually prove the following stronger version of Theorem \ref{thA}, via the definition of \emph{Upper Banach generic (UB-generic) sets}; cf.\ Definition~\ref{str} and Theorem~\ref{cor:epwp}.

\begin{theor}\label{thnew}
   Let $\Gamma$ be  a finitely generated group. 
If there exists a finite set of generators $X$ of $\Gamma$ such that the Equality Problem is solvable on a set $S\times S$,  where $S\subset \mathbb F_X$ is UB-generic, then $\Gamma$ has solvable Word Problem.
  \end{theor}
This suggests that these new densities might  be interesting \emph{per se}: we investigate the class of groups having solvable WP on UB-generic sets and characterize them as follows, cf.~Theorem~\ref{principale}.
\begin{theor}\label{va}
  Let $\Gamma$ be  a group generated by a finite set $X$. Then
$\Gamma$ has  solvable UB-generic WP with respect to $X$ if and only if 
there exists an infinite computably enumerable sequence  of words in $\mathbb F_X$ representing elements of strictly increasing length in $\Gamma$.
\end{theor}
A first consequence of this characterization is that solvability of UB-generic WP does not depend on the choice of the finite generating set. Moreover the class of groups whose \emph{WP  is unsolvable on every UB-generic set} can only contain torsion groups. Another straightforward consequence of Theorem~\ref{va} is the following, cf. Corollary~\ref{pra}.
 \begin{cor}\label{thB}
Let $\Gamma$ be an infinite algorithmically finite group generated by a finite set $X$. The WP of $\Gamma$  is unsolvable on every UB-generic set.
\end{cor}
Due to the exotic nature of algorithmically finite groups, we feel that their inclusion in a broader class of groups with nice and diverse characterizations can be helpful (cf.~Theorem~\ref{principale}).
Moreover, since  UB-genericity is a weaker notion than  classic genericity,  we prove that Dehn monsters also constitute the \emph{first example of computably presented groups with unsolvable generic Word Problem}.   This was obviously among the purposes of~\cite{Myas/11}, but there the emphasis was on the Equality Problem. In light of our results, it seems appropriate to turn the attention to partial WP, or at least to consider a different definition for partial EP.
In fact, proving Theorems~\ref{thA} and \ref{thnew} has required an analysis of the connection between Equality and Word Problem, which had sometimes been previously considered,  but not deeply unraveled. This analysis revealed that the odd behavior exhibited in Theorem~\ref{thA} is essentially due to the particular way of defining solvability of generic EP via \emph{Fubini-genericity}: taking a more classical definition such as the one outlined in~\cite{generic02}, we prove the expected equivalence between the two generic problems in Theorem~\ref{thm:wpep}.
\begin{theor}
Let $\Gamma$ be a group generated by a finite set $X$. Then $\Gamma$ has solvable generic EP with respect to $X$ in the sense of Definition~\ref{defEP} if and only if it has solvable generic WP with respect to $X$.
\end{theor}
Even if decidability of the generic EP is equivalent to that of the generic WP, we do not know if this is the case for the complexity: the partial relations obtained by our proofs are presented in Corollary~\ref{cor:comp} (see Remark~\ref{rem-epwp-comp2} for the analogous discussion about Theorem~\ref{thA}).

We devote the final part of this paper to  taking a unifying look at these new and old classes of groups, defined according to the increasing level of (un)solvability of the partial WP, asking a few questions on the still unknown relations among them.

\subsection{Notation and preliminaries}\label{notation}
Throughout this paper, $\Gamma$ is a finitely generated group. For a finite set of generators $X$ we set $\lvert X\rvert =d$. We denote by $\pi :\mathbb F_X \to \Gamma$ the canonical epimorphism from the free group on $X$  to $\Gamma$.  The normal subgroup $\ker \pi\lhd \mathbb F_X$ of \emph{trivial} words is often  called the \emph{Word Problem} of $\Gamma$.
 We denote by $\lvert g \rvert_\Gamma$ the word length of $g$ in $\Gamma$
  with respect to $X$, for $\omega\in \mathbb F_X$  
  we simply write $|\omega|$ instead of $|\omega|_{\mathbb F_X}$. 
 Note that $\lvert g \rvert_\Gamma=\min\{\lvert \omega\rvert: \omega\in\mathbb F_X, \pi(\omega)=g\}$. For the $k$-th direct power $\mathbb F_X^k$ of the free group $\mathbb F_X$ we will consider the usual generators, so that    $|(\omega_1,\ldots, \omega_k)|_{\mathbb F_X^k}=|\omega_1|+\cdots+|\omega_k|$.
  We denote with~$S_n(\Gamma)$ the sphere  and  with~$B_n(\Gamma)$ the
  ball of radius $n$ in $\Gamma$, respectively. For the free group we
  simply write $S_n$  instead of $S_n(\mathbb F_X)$ and $B_n$ instead
  of $B_n(\mathbb F_X)$. We let $e$ denote the empty word in  $\mathbb F_X$.

A set $S\subset \mathbb F_X^k$ is \emph{negligible} if
\begin{equation}\label{def:gen}
  \lim_{n\to\infty}\frac{\lvert S\cap B_n(\mathbb F_X ^k)\rvert}{\lvert B_n(\mathbb F_X ^k)\rvert}=0,\end{equation}
and \emph{generic} if its complement in $\mathbb F_X ^k$ is negligible. The set $S$ is \emph{exponentially negligible} if it is negligible and the convergence in~\eqref{def:gen} is exponential, that is 
$\beta^n\frac{\lvert S\cap B_n(\mathbb F_X ^k)\rvert}{\lvert B_n(\mathbb F_X ^k)\rvert}\to 0$, for some $\beta>1$.
We say in this case that the complement is    \emph{exponentially generic} (see also~\cites{generic02,gilman/10}).
We will  call \emph{Fubini-generic} the special generic subsets of $\mathbb F^2 _X$ of the form  $S\times S \subset  \mathbb F^2_X$, with $S$ generic in~$ \mathbb F_X$.

Setting $\alpha:=2d-1$, an easy computation ensures that, when $d>1$, there exist positive constants $c_2,c_s,C_1,C_2$ such that 
\begin{equation}\label{palla}
\begin{split}
|S_n|= c_s \alpha^n, \quad& \qquad
 \alpha^n\leq |B_n|\leq C_1 \alpha^n,\\
 c_2 (n+1) \alpha^n\leq & |B_n(\mathbb F_X^2)|\leq C_2 (n+1) \alpha^n.
\end{split}
\end{equation}
This implies, together with Cesaro-Stolz, that $S\subset \mathbb F_X$ is (exponentially) negligible
if and only the sequence $\frac{|S\cap S_n|}{|S_n|}$ (exponentially)
tends to zero. This sequence is sometimes used in the literature to define, in an equivalent way, generic/negligible subsets of $\mathbb F_X$.

 For definitions and basic facts on algorithms we refer to~\cite{Cooper}. The group $\Gamma$ has \emph{solvable Word Problem (WP) on a subset $S\subset \mathbb F_X$} if there exists a partial algorithm that stops at least for every $\omega\in S$, and, if it stops, it establishes whether $\omega$ is trivial or not.
The group $\Gamma$ has \emph{solvable WP} with respect to $X$ if it has solvable WP on  $\mathbb F_X$; it has \emph{solvable generic WP} with respect to $X$ if it has solvable WP on a generic subset  $S\subset \mathbb F_X$. The group $\Gamma$ has \emph{solvable Equality Problem (EP) on a subset $T\subset \mathbb F^2_X$}  if there exists a partial algorithm that stops at least for every $(\omega_1,\omega_2)\in T$, and establishes if $\pi(\omega_1)=\pi(\omega_2).$ The group $\Gamma$ has \emph{ solvable Fubini-generic EP} with respect to $X$ if it has solvable EP on a Fubini-generic subset of~$\mathbb F^2_X$; notice that this is exactly the definition of generic EP in the sense of~\cite{Myas/11}.

All the previous generic problems can be stated in  \emph{exponentially
  generic} versions. Also, in all the previous problems we can replace
\emph{solvability} with \emph{solvability in a (time) complexity class
  $C$}, see \cite{generic02} for formal definitions and details.
Notice that we will work with freely reduced words, that is words in
$\mathbb F_X$ instead of $(X\cup X^{-1})^*$. This is not relevant in our setting, since we  assume, from now on, that  complexity classes are defined by
collections of complexity bounds which are  also closed  under addition of linear functions.

Not much is known about the invariance of generic solvability problems under change of the finite generating set $X$ (see  \cite[Section~3]{generic02}). Clearly solvability of the whole WP does not depend on $X$; as a consequence,  thanks to our Theorem \ref{thA}, solvability of Fubini-generic EP does not depend on the choice of the finite set of generators either.

\begin{acknow}
  Both authors gratefully acknowledge support from the Erwin Schr\"odinger International Institute for Mathematics and Physics at the University of Vienna in the form of  Junior Research Fellowships. 
  AC was partially supported by the Irish Research Council through grant no.\ GOIPD/2018/319.
  MC would also like to thank the Faculty of Mathematics at the University of Vienna for the warm hospitality.
  We thank Goulnara Arzhantseva, Christopher Cashen, Tullio
  Ceccherini-Silberstein, Ilya Kapovich, Tobias Rossmann and  Paul Schupp for
  mathematical discussions and helpful comments. We are grateful to
  the referee for comments and suggestions which helped improving this manuscript.
\end{acknow}


\section{Upper Banach generic Word Problem}\label{sec:ubgwp}
 We give our main  definition concerning densities of subsets of free groups. 
\begin{definition}\label{str}
Let $S$ be a subset of $\mathbb F_X$. We define the Lower and the Upper Banach densities ($\underline\mu$ and $\overline \mu$, respectively) of $S$ as:
\begin{align*}
\underline{\mu}(S):=\liminf_{n\to\infty}  \min_{ \omega\in \mathbb F_X }  \frac{|S\cap \omega B_n|}{|B_n|},&\quad \overline{\mu}(S):=\limsup_{n\to\infty}  \max_{ \omega\in \mathbb F_X } \frac{|S\cap \omega B_n|}{|B_n|}.
\end{align*}
A set $S\subset\mathbb F_X$ is \emph{Upper Banach generic} (\emph{UB-generic} for short) if $\overline{\mu}(S)=1$. A set $N\subset\mathbb F_X$ is \emph{UB-negligible} if its complement is UB-generic, that is if $\underline \mu(N)=0$. Similarly, a set $S\subset\mathbb F_X$ is \emph{Lower Banach generic} (\emph{LB-generic} for short) if $\underline{\mu}(S)=1$ and $N\subset\mathbb F_X$ is \emph{LB-negligible} if its complement is LB-generic, that is if $\overline \mu(N)=0$.
\end{definition}

The following proposition characterizes UB-generic subsets of free groups as those containing translates of any balls, and thus of any finite sets. 
\begin{propo}\label{prima}
A subset $S\subset \mathbb F_X$ is UB-generic if and only if for all $n\in \mathbb N$ there exists $\omega_n\in \mathbb F_X$ such that $\omega_n B_n\subset S$.
\end{propo}
\begin{proof}
The existence of a sequence $\{\omega_n\}_{n\in\mathbb N}\subset \mathbb F_X$ such that  $\omega_n B_n\subset S$ for all $n$ clearly implies the UB-genericity of $S$. For the converse, suppose $S$ is UB-generic. We denote by $N:=S^c$ the complement of $S$, which is UB-negligible. Suppose by contradiction that there exists $k\in\mathbb N$ such that $\omega B_k\not\subset S$ (equivalently, 
$\omega B_k\cap N\neq \emptyset$) for all $\omega\in \mathbb F_X$. One can check that (see, for instance, \cite[Lemma~5.3]{Cavaleri/18}), for $n$ big enough, the ball of radius $n$ contains $|S_{n-2k}|$ disjoint translates of $B_k$:
$$B_n\supset \bigsqcup_{i=1}^{|S_{n-2k}|} \omega_i B_k,$$
and then, for every $\omega\in\mathbb F_X$ we have 
$\omega B_n\supset \bigsqcup_i \omega\omega_i B_k.$
Since $N$ contains at least a word for each translate of $B_k$, we have $|N\cap \omega B_n| \geq |S_{n-2k}|$, independently of~$\omega$. Then if $d>1$, by Equation \eqref{palla}
$$ \min_{\omega\in \mathbb F_X}  \frac{|N\cap \omega B_n|}{|B_n|}  \geq \frac{ |S_{n-2k}|}{|B_n|}\geq  \frac{ c_s \alpha^{n-2k}}{C_1\alpha^n}>0,$$ that is impossible since $\underline \mu(N)=0$. The case $d=1$ is actually a classical result (see \cite[Lemma~1]{sumset}), in our setting it is enough to notice that $\omega B_n$ contains  $\lfloor \frac{n}{k}\rfloor\sim \frac{|B_n|}{2k}$ disjoint translates of $B_k$.
\end{proof}

\begin{remark}\label{oss} It follows from the above proposition that if  $S$ is UB-generic, then 
 $$ S^{-1}S\supset (\bigcup \omega_n B_n)^{-1} (\bigcup \omega_n B_n)\supset \mathbb F_X.$$ It is clear from  Definition~\ref{str} that  UB-genericity is weaker than genericity, which is in turn   weaker than LB-genericity. For a fixed non-trivial word $\omega \in\mathbb F_X$ and some $f:\mathbb N \to \mathbb N$, define   ${T_f:=\bigcup_{n=1}^\infty \omega^{f(n)}B_n}$  (analogous sets were  considered, for instance, in  \cite[Remark 5.4]{Cavaleri/18}). 
 The set  $T_f$ is always UB-generic but, choosing $f$ growing fast enough, it is also negligible. Conversely,  the set $T^c_f$ is  always  non-LB-generic and, for the chosen $f$, it is also generic.
 Moreover, it is an easy exercise to define a set $S$ such that $S^{-1}S=\mathbb F_X$ and $|S\cap S_n|\leq 1$ for all $n$. The last property ensures that such a set is not UB-generic. To summarize,  the following hold.
$$ S  \mbox{ LB-generic } \quad_{\not\Longleftarrow }^{\Longrightarrow}   \quad S  \mbox{ generic } \quad _{\not\Longleftarrow }^{\Longrightarrow} \quad S  \mbox{ UB-generic }\quad _{\not\Longleftarrow }^{\Longrightarrow}  \quad S^{-1}S=\mathbb F_X.$$
\end{remark}

\begin{definition}\label{def:UB} Let $\Gamma$ be a group generated by a finite set $X$. We say that  $\Gamma$ has \emph{solvable UB-generic  WP} with respect to $X$ if it has solvable WP on an UB-generic subset of $\mathbb F_X$.
\end{definition}
 Note that our Theorem \ref{principale} implies, a posteriori, that this definition does not depend on the choice of the finite set of generators; cf. Corollary \ref{coco}.

It is well known that if $\Gamma$ is infinite, the set $\ker \pi \subset \mathbb F_X$, i.e.\ the set of the Word Problem, is negligible. For this reason, in order to study  generic Word Problem, one can restrict the attention to the behavior of an algorithm on the non-trivial words. On the other hand, in the investigation of UB-generic WP, the negligibility of $\ker \pi$  is not enough, essentially because the intersection of an UB-generic set and a generic set can  even be empty (e.g.\ the sets $T_f$ and $T^c_f$ in Remark~\ref{oss}). This is not the case for the intersection of an UB-generic set and a LB-generic set: one can easily check that this intersection is always UB-generic. The next theorem establishes that, if $\Gamma$ is infinite, the set of trivial words is not only negligible, but also LB-negligible, thus ensuring that a set $S$ is UB-generic if and only if $S\setminus \ker \pi$ is UB-generic.

\begin{theo}\label{osse}
Let  $\Gamma$ be an infinite  group  generated by a finite set $X$ and  let $\pi:\mathbb F_X\to \Gamma$ denote the canonical projection.   Then $\overline{\mu}(\ker \pi)=0$. Equivalently, $\ker \pi$ is LB-negligible.  \end{theo}
\begin{proof}
If $X$ consists of a single generator the claim is clear since $\Gamma$ must be cyclic and $\ker \pi$ is trivial; we assume $d>1$.
Let $\omega \in\mathbb F_X$ with $\pi(\omega)=g$. Note that if  $\omega' \in\mathbb F_X$ is such that $\pi(\omega')=g$, then $\lvert \ker \pi\cap  \omega S_n \rvert=\lvert \ker \pi\cap \omega' S_n \rvert$, and thus this quantity does not depend on the choice of representatives of $g$.
Denoting by  $\gamma\in (\sqrt{2d-1}, 2d-1]$ the \emph{cogrowth} of $\Gamma$ (cf.~\cite{Gri/77}), the ratio $\frac{\lvert \ker \pi\cap \omega S_n \rvert }{\gamma^n}$ tends to zero  uniformly for $g\in \Gamma$ (cf.~\cite[Theorem 2]{Woess/83}).
Therefore, we have  $\lim_{n\to\infty}{\max_{\omega\in \mathbb F_X}  \frac{|\ker\pi \cap \omega S_n|}{|S_n|}}= 0$, since 
\begin{equation}\label{cogrowth}
\frac{\lvert \ker \pi\cap \omega S_n \rvert }{\lvert S_n\rvert}=\frac{\lvert \ker \pi\cap \omega S_n \rvert }{\gamma^n} \frac{\gamma^n}{|S_n|}
\end{equation}
and, being $\gamma\leq 2d-1$, the sequence $\left\{\frac{\gamma^n}{|S_n|}\right\}$ is uniformly bounded.

  By Cesaro-Stoltz, also $$\lim_{n\to\infty}\frac{\sum_{i=0}^n \max_{\omega\in \mathbb F_X}  |\ker\pi \cap \omega S_i|}{\sum_{i=0}^n {|S_i|}}= 0.$$    
  The result follows, as $$\max_{\omega\in \mathbb F_X}  \frac{|\ker\pi \cap \omega B_n|}{|B_n|}\leq \frac{\sum_{i=0}^n \max_{\omega\in \mathbb F_X}  |\ker\pi \cap \omega S_i|}{\sum_{i=0}^n {|S_i|}}.$$ \qedhere
\end{proof}
\begin{remark}
If $\Gamma$ is infinite, by combining  \cite[Theorem 2]{Woess/83}, Equation \eqref{cogrowth} and the cogrowth criteria \cites{Gri/77,Cohen/82}, we get that $\ker \pi$ is \emph{exponentially LB-negligible} if and only if $\Gamma$ is non-amenable.
  On the other hand, an UB-negligible set $S$ is always \emph{exponentially UB-negligible}, since in  light of Proposition~\ref{prima} the sequence $\left\{ \min_{ \omega\in \mathbb F_X }  \frac{|S\cap \omega B_n|}{|B_n|}\right\}_{n\in\mathbb N}$ of Definition~\ref{str}  is eventually~$0$.
\end{remark}

For our comparisons between Equality and Word Problems,  we need to switch between subsets of $\mathbb F_X ^2$ and of $\mathbb F_X $. To this purpose we define the following map
\begin{equation}\label{eq:p}
  \tau \colon \mathbb F_X\times \mathbb F_X \to \mathbb F_X,\,(\omega_1,\omega_2)\mapsto \omega_1^{-1}\omega_2.\end{equation}

\begin{lemma}\label{lem-epwp} Let $\Gamma$ be a group generated by a finite set $X$. Then  
$\Gamma$ has solvable Equality Problem on a set $T\subset \mathbb F_X\times  \mathbb F_X$ if and only if $\Gamma$ has solvable Word Problem on the set $\tau(T)\subset \mathbb F_X$.
\end{lemma}
\begin{proof}
Let us denote with $\mathcal A$ the algorithm solving the Equality Problem on $T$, we are going to describe an algorithm solving the Word Problem on the set $\tau(T)$. 
For every $\omega, v \in\mathbb F_X$, the word $\omega$ is trivial if and only if   $\pi(v)=\pi(v\omega)$.
 Let us denote by $\{v_n\}_{n\in \mathbb N}$
  a computable enumeration of $\mathbb F_X$.
The algorithm takes $\omega$ as an input and runs the algorithm  $\mathcal A$ simultaneously on the pairs
 $(v_n , v_n \omega)$: if there exists 
 $n\in \mathbb N$ such that 
 $( v_n, v_n \omega)\in T$, then the algorithm stops establishing if
  $\omega$ is in $\ker \pi$ or not. We conclude observing that 
$\{\omega: \exists\, v\in \mathbb F_X: (v , v \omega)\in T\}=\tau(T)$.

For the other direction, it is easy to see that solvability of WP on a set $S\subset \mathbb F _X$  implies solvability of EP on all pairs of words having image through $\tau$ in $S$ itself, namely $\tau^{-1}(S)$. But if $S=\tau(T)$ for some $T\subset \mathbb F^2_X$ then $T\subset \tau^{-1}(\tau(T))$.
 \end{proof}
For $T\subset \mathbb F _X^2$, if we denote with $\overline T:= \tau^{-1}(\tau(T))$,  solvability of EP on $T$ implies solvability of EP on $\overline T$. For $S\subset \mathbb F_X$, if we also denote, by slight abuse of notation, $\overline S:=\{\alpha^{-1} \omega \alpha:\, \omega\in S, \alpha\in \mathbb F_X\}=S^{\mathbb F_X}$, solvability of WP on $S$ implies solvability of WP on $\overline S$. These facts can be also directly proved noticing that $${\overline T\subset\{(\alpha \omega_1 \beta,\alpha \omega_2 \beta):\, (\omega_1,\omega_2)\in T, \alpha, \beta\in \mathbb F_X\}}.$$
Even if  solvability of the EP on $T$ is equivalent to that on $\overline T$, the complexity classes of these two problems are, a priori, different.

\begin{remark}\label{rem-epwp-comp}
Let $\Gamma$ be a group generated by a finite set $X$. 
If the WP on $\tau(T)\subset \mathbb F_X$ is solvable in some complexity class $C$ then the EP on $T\subset \mathbb F_X^2$ is solvable in the complexity class $C$.
It follows in fact from the proof of the previous lemma that the algorithm solving the EP on $T$ is induced by the algorithm solving the WP on $\tau(T)$ without altering  the complexity.
If the EP on  $T\subset \mathbb F_X^2$ is solvable in the complexity class $C$ and $T=\overline T$, then  WP on $\tau(T)\subset \mathbb F_X$ is solvable in the complexity class $C$. Indeed, if $T=\overline T$, then $(e,\omega)\in T$ for every $\omega\in \tau(T)$ so the steps performed to establish if $\omega$ is trivial are exactly those of $\mathcal A$ on $(e,\omega)$.
\end{remark}

\begin{theorem}\label{cor:epwp}Let $\Gamma$ be  a group generated by a finite set $X$. 
If $\Gamma$ has solvable Equality Problem on a set $S\times S$, where $S\subset \mathbb F_X$ is UB-generic, then $\Gamma$ has solvable Word Problem. 
\end{theorem}
\begin{proof}
By virtue of Remark~\ref{oss}, if $S$ is UB-generic then $\tau(S\times S)=\mathbb F_X$. By Lemma \ref{lem-epwp}, the group $\Gamma$ has solvable Word Problem.
\end{proof}

\begin{proof}[Proof of Theorem~\ref{thA}]
 Suppose that  $\Gamma$ has Fubini-generic solvable EP, that is  the EP is solvable on a set $S\times S$ with $S\subset \mathbb F_X$ generic. It follows from Remark~\ref{oss} that the subset $S$ is also UB-generic. By virtue of  Theorem~\ref{cor:epwp} the group $\Gamma$ has solvable~ WP.\end{proof}

As a consequence, while considering the Equality Problem on ``small'' square subsets $S\times S \subset \mathbb F^2_X$ produces
new concepts such as the  algorithmic finiteness, considering its Fubini-generic solvability is simply equivalent to solvability of the classical WP, from the point of view of pure decidability. However, Fubini-genericity is still worth investigating in connection with complexity and with other decision problems, as we observe in the following remarks.
\begin{remark}
  Note that an analogue of Theorem~\ref{thA} for the Conjugacy Problem cannot exist. Indeed, in~\cite{BMR09,BMR07} it is proved that, under suitable hypotheses on $H$, the Miller groups $G(H)$ have solvable Conjugacy Problem on exponentially Fubini-generic (and in fact even bigger) sets. On the other hand, they have unsolvable Conjugacy Problem. \end{remark}
  
\begin{remark}\label{rem-epwp-comp2}
Even if  solvability of the WP on $\Gamma$  is equivalent to
solvability of the EP on $S\times S$  for any UB-generic set $S\subset
\mathbb F_X$, their complexity classes are, a priori, different and possibly depend on the shape of $S$. Indeed, Proposition~\ref{prima}  only ensures that for an UB-generic set $S$ the function $U_S(n):=\min\{|\omega|: \omega B_n \subset S\}$ is finite-valued but the natural upper bound for the complexity of our algorithm (described in the proof of Lemma~\ref{lem-epwp}) depends on the growth rate of $U_S(n)$.
\end{remark}

One might still want to investigate the partial Equality Problem. A possible way to do it is to employ  the natural definition of genericity in products
(as already defined in~\cite{generic02}). As we show in the next
section, considering the  generic EP with this notion of genericity does not lead to new behavior either.


\section{Generic Word Problem and generic Equality Problem}
\begin{definition}\label{defEP}
We say that a group $\Gamma$ generated by a finite set $X$ has \emph{ solvable generic EP} with respect to $X$, if it has solvable EP on a generic set of $\mathbb F_X^2$ (in the sense of Equation \eqref{def:gen}).
\end{definition}

\begin{theorem}\label{thm:wpep}
  Let  $\Gamma$ be a group generated by a finite set $X$. Then $\Gamma$ has solvable generic Word Problem with respect to $X$ if and only if it has solvable generic Equality Problem with respect to $X$. Moreover, the Word Problem on $\Gamma$ is solvable only on negligible subsets of $\mathbb F_X$ if and only if  the Equality Problem on $\Gamma$ is solvable only on negligible subsets of $\mathbb F_X^2$.
\end{theorem}
An easy consequence of the previous theorem and Remark \ref{rem-epwp-comp} is the following.
\begin{corollary}\label{cor:comp}
Let $\Gamma$ be a group generated by a finite set $X$.
If the generic WP  with respect to $X$   is solvable in some complexity class $C$ then  the generic EP with respect to $X$ is solvable in the complexity class $C$. Vice versa, if the EP on a generic set $T\subset \mathbb F_X^2$, with $T=\overline T$, is solvable in some complexity class $C$, then the generic WP  with respect to $X$   is solvable in the complexity class $C$.
\end{corollary}
To prove Theorem \ref{thm:wpep} we will need  two lemmas to compare densities of subsets of $\mathbb F _X $ and of $\mathbb F_X\times\mathbb F_X$. Recall that $|X|=d$ and $\alpha=2d-1$, and recall the definition of the map $\tau:\mathbb F_X\times \mathbb F_X\to \mathbb F_X$ in Equation~\eqref{eq:p}.
For $s\in \mathbb F_X$ and $n\in \mathbb N$, we define the subset $$P(s,n):=\{\omega\in \mathbb F_X: |\omega|+|\omega s|\leq n\},$$
which we will need to estimate the density of preimages under $\tau$.
\begin{lemma}\label{nuovo}
For every reduced word $s=s_k s_{k-1} \cdots s_1\in \mathbb F_X$, $s_i\in X\cup X^{-1}$, with $|X|>1$, we have 
$$ P(s^{-1},n)=       \bigcup^k_{i=0} B_{\frac{n-k}{2}} s_i\cdots s_1.$$
\end{lemma}
\begin{proof}
  Informally,  we want to show that in the left Cayley graph of $\mathbb F_X$ the set $P(s^{-1},n)$  is the ``neighborhood of radius $\frac{n-k}{2}$''  of the geodesic from the empty word to $s$.

  Let   $\omega\in P(s^{-1},n)$. Suppose that the product of $\omega$ with $s^{-1}=s_1^{-1}\cdots s_k^{-1}$ induces exactly
 $i$ cancellations: this implies that $\omega=v s_i \cdots s_1$  for some $v$ not ending in $s_i^{-1}$ or  $s_{i+1}$ and $|\omega s^{-1}|=|v
 s_{i+1}^{-1}\cdots s_k^{-1}|=|v|+k-i$. Since $\omega\in P(s^{-1},n)$ we have
 $$n\geq|\omega|+|\omega s^{-1}|=|v|+i+|v|+k-i,$$
 that is $|v|\leq \frac{n-k}{2}$ and therefore  $ P(s^{-1},n) \subset           \bigcup^k_{i=0} B_{\frac{n-k}{2}} s_i\cdots s_1$.
  
  To prove the other inclusion, let now $\omega \in  \bigcup^k_{i=0} B_{\frac{n-k}{2}} s_i\cdots s_1 $. This means that there exist  $i$ with $0\leq i\leq k$ and $v\in B_{ \frac{n-k}{2}}$ such that  $\omega=v s_i\cdots s_1$. Then 
$$|\omega|+|\omega s^{-1}|=|v s_i\cdots s_1|+|v s_{i+1}^{-1}\cdots s_k^{-1}|  \leq  2 \left( \frac{n-k}{2}\right)+i+k-i\leq n,$$
  that is $\omega\in P(s^{-1},n)$.
\end{proof}

\begin{lemma}\label{lem:fg}
  For any $S\subset \mathbb F_X$, for any $T \subset \mathbb F^2_X$, the following hold:
  \begin{enumerate}
    \item
      $S$ is (exponentially) negligible if and only if $\tau^{-1}(S)$ is (exponentially) negligible;
    \item
  $S$ is (exponentially) generic if and only if $\tau^{-1}(S)$ is  (exponentially) generic;
       \item 
   if  $T$ is (exponentially) generic then $\tau(T)$ is (exponentially) generic;
  \item 
 if  $\tau(T)$ is (exponentially) negligible then  $ T$ is (exponentially) negligible.

  \end{enumerate}
\end{lemma}

\begin{proof}
When $d=1$ all claims follow from  simple computations in the lattice~$\mathbb Z^2$. 

Suppose $d>1$. Observe that 
\begin{align}\label{neglitubi}
|\tau^{-1}(S)\cap B_n(\mathbb F^2_X)|&=  \sum_{s\in S} |\{(\omega_1,\omega_2): |\omega_1|+|\omega_2|\leq n,\, \omega_1^{-1}\omega_2=s  \} |\\&= \sum_{s\in S}  |P(s,n)| \nonumber.
\end{align}
Notice that, if $|s|>n$, the set $P(s,n)$ is empty. 
We are going to prove that if $S$ is (exponentially) negligible, then $\tau^{-1}(S)$ is (exponentially) negligible.\\
By virtue of  Equation \eqref{palla} and Lemma \ref{nuovo}, we have, for any  $s\in S_k$,  $|P(s,n)|\leq (k+1) |B_{(n-k)/2}|\leq C_1  (k+1) \alpha^{(n-k)/2}$. As a consequence
\begin{align*}
\frac{|\tau^{-1}(S)\cap B_n(\mathbb F_X^2)|}{|B_n(\mathbb F_X^2)|}&\leq   \frac{\sum_{s\in S}  |P(s,n)|}{c_2(n+1)\alpha^n}
\leq \frac{C_1}{c_2} \sum_{k=0}^{n} \frac{|S\cap S_{k}|}{\alpha^{\frac{n+k}{2}}}.
\end{align*}
Assume that $\beta\geq 1$ is such that $\beta^n\frac{|S\cap
  B_n|}{|B_n|}\to 0$. Note that this is equivalent to  $S$  being
negligible when  $\beta=1$, and $S$ being  exponentially negligible when $\beta>1$.
Without loss of generality we can also assume $\beta< \sqrt{\alpha}$. 
By Equation \eqref{palla}, we have that  $\beta^n\frac{|S\cap S_n|}{\alpha^n}\leq \beta^n\frac{|S\cap B_n|}{\alpha^n}\to 0$. 
\\In particular, for every $\varepsilon>0$, there exists $\bar{k}$ such that $\beta^k\frac{|S\cap S_k|}{\alpha^k}<\varepsilon$ for every $k> \bar{k}$. Moreover, for every fixed $\bar{k}$ there exists $\bar{n}\in \mathbb N$ such that $\bar{n}>\bar{k}$ and
$\left(\frac{\beta}{\sqrt{\alpha}}\right)^{\bar{n}}\sum_{k=0}^{\bar{k}} \frac{|S\cap S_{k}|}{\alpha^{k/2}}<\varepsilon$.
For any $n>\bar{n}$ we have
\begin{align*}
\beta^n \frac{|\tau^{-1}(S)\cap B_n(\mathbb F_X^2)|}{|B_n(\mathbb F_X^2)|}
 &\leq\frac{C_1}{c_2} \left( \frac{\beta^n }{\alpha^{n/2}}\sum_{k=0}^{\bar{k}} \frac{|S\cap S_{k}|}{\alpha^{k/2}} + \sum_{k=\bar{k}+1}^{n} \beta^k\frac{|S\cap S_k|}{\alpha^k}  \left(\frac{\beta}{\sqrt{\alpha}}\right)^{n-k}\right)\\
&\leq \frac{C_1}{c_2} \left( \left(\frac{\beta}{\sqrt{\alpha}}\right)^{\bar{n}}\sum_{k=0}^{\bar{k}} \frac{|S\cap S_{k}|}{\alpha^{k/2}} +\varepsilon \sum_{k=\bar{k}+1}^{n}\left(\frac{\beta}{\sqrt{\alpha}}\right)^{n-k}\right)\\
&\leq \frac{C_1}{c_2}\left(\varepsilon+\varepsilon \sum^{\infty}_{t=0} \left(\frac{\beta}{\sqrt{\alpha}}\right)^t \right)\leq  \varepsilon  \frac{C_1}{c_2} \left(1+\frac{\sqrt{\alpha}}{\sqrt{\alpha}-\beta}\right).
\end{align*}
Therefore  $\tau^{-1}(S)$ is negligible, and if $\beta>1$ even exponentially negligible.

Now suppose that $S$ is not (exponentially) negligible. By Lemma \ref{nuovo} we have that $|P(s,n)|=n+1$ for $s\in S_n$.
Combining with Equation \eqref{neglitubi} and Equation \eqref{palla} we have
\begin{equation*}
 \frac{|\tau^{-1}(S)\cap B_n(\mathbb F_X^2)|}{|B_n(\mathbb F_X^2)|}\geq \frac{\sum_{s\in S_n}  |P(s,n)|}{C_2(n+1) \alpha^n }\geq  \frac{1}{C_2  } \frac{|S\cap S_n|}{ \alpha^n} 
\end{equation*}
proving that $\tau^{-1}(S)$ is not (exponentially) negligible. This completes the proof of~(1).

The second claim easily follows from (1) by taking complements and observing that $\tau^{-1}(S^c)=\left(\tau^{-1}(S)\right)^c$.
Finally, (3) and (4) easily follow by taking $S=\tau(T)$ and observing that $T\subset \overline T=\tau^{-1}(\tau(T))$.
\end{proof}

\begin{proof}[Proof of Theorem~\ref{thm:wpep}]
  Suppose $\Gamma $ has solvable WP on $S\subset \mathbb F_X$ generic. This means, by virtue of Lemma~\ref{lem-epwp}, that $\Gamma$ has solvable EP on $\tau^{-1}(S)$, which by Lemma~\ref{lem:fg} is generic.
For the converse, suppose $\Gamma $ has solvable EP on a generic set  $T\subset \mathbb F^2_X$. By virtue of Lemma~\ref{lem-epwp}, the group $\Gamma$ has solvable WP on $\tau(T)$, which by Lemma~\ref{lem:fg} is generic. 

Suppose that solvability of the WP on $S\subset \mathbb F_X$ implies that $S$ is negligible. Assume that $\Gamma$ has solvable EP on $T$ with $T$  non-negligible. Then, by  Lemma~\ref{lem:fg} the subset $\tau(T)$ is non-negligible and by Lemma~\ref{lem-epwp} the group $\Gamma$ has solvable WP on $\tau(T)$, that is a contradiction. An analogous argument proves that if  EP is solvable only on negligible sets, then the same is true for the WP.
\end{proof}

\begin{remark}   
Clearly, Theorem~\ref{thm:wpep}  holds true replacing generic (resp.\ negligible) with exponentially generic (resp.\ exponentially negligible). 

In  \cite{gilman/10} (and elsewhere),  the length in $\mathbb F_X\times \mathbb F_X$ used to define genericity is $\ell(\omega_1,\omega_2):=\max\{|\omega_1|,|\omega_2|\}$.
The ball of radius $n$ in $\mathbb F_X^2$ with respect to this length is $B_n\times B_n$. An analogue of Lemma \ref{lem:fg} (and therefore of Theorem \ref{thm:wpep}) can be proved in this setting by using similar arguments. The role of the set $P(s,n)$ is played, in that case, by the set $B_n\cap B_n s$. One can indeed show, similarly to Lemma \ref{nuovo}, that
$B_n\cap B_n s$ is the ball of radius $n-\frac{|s|}{2}$ centered in the middle point of the geodesic from the empty word to $s$.
Moreover, this analogue of Lemma \ref{lem:fg}  gives a generalization of \cite[Lemma 3.2]{gilman/10}  which holds in non-exponential setting.
\end{remark}

\vspace{3mm}


\section{Upper Banach generic Word Problem and algorithmic finiteness}

A group $\Gamma$  generated by a finite set $X$ is algorithmically finite if  there does not exist a computable enumeration of an infinite set of words in $\mathbb F_X$  projecting onto pairwise distinct elements of $\Gamma$, or, equivalently if (cf.\ \cite{Myas/11}):
\begin{itemize}
\item
solvability of EP on $S\times S$ implies that $\pi(S)$ is finite;
\item
for any infinite computably enumerable set $S\subset \mathbb F_X$ we have that $S^{-1}S\cap \ker \pi\neq e$;
\item
solvability of WP on $S^{-1}S$ implies that $\pi(S)$ is finite (Lemma \ref{lem-epwp}).
\end{itemize}
Note that  algorithmic finiteness does not depend on the choice of the
finite generating set, cf.~\cite[Lemma~2.4]{Myas/11}.

We now  characterize, in a similar fashion, solvability of the UB-generic WP (see Definition~\ref{def:UB}), and therefore groups without this property.
\begin{theo}\label{principale}
  Let $\Gamma$ be an infinite group generated by a finite set $X$. The following are equivalent:
\begin{enumerate}
\item
$\Gamma$ has  solvable UB-generic WP with respect to $X$;
\item
there exists $S\subset \mathbb F_X$  computably enumerable, UB-generic and such that $S\cap \ker\pi=\emptyset$;
\item
there exists a computably enumerable sequence $\{\omega_n \}_{n\in \mathbb N}\subset \mathbb F_X$ such that $|\omega_n|_\Gamma>n$ for all $n\in \mathbb N$;
\item
there exists a computably enumerable sequence $\{\omega_n \}_{n\in \mathbb N}\subset \mathbb F_X$ such that $|\omega_{n+1}|_\Gamma>|\omega_{n}|_\Gamma$ for all~$n\in \mathbb N$.
\end{enumerate}
\end{theo}
\begin{proof}${}$\\
$(1)\iff(2)$\\
Suppose $\Gamma$ has solvable  UB-generic WP, then there exists an algorithm solving the Word Problem  on an UB-generic subset of inputs $S'\subset \mathbb F_X$. 
The subset $S:= S'\setminus \ker\pi$ is computably enumerable and, by virtue of  Theorem \ref{osse}, UB-generic. Vice versa, the computable enumeration of $S$ solves the WP on the UB-generic set $S$. \\
$(2)\iff(3)$\\
Suppose that  $S$ is a computably enumerable UB-generic subset of non-trivial words. By an elementary argument, the subset $\Omega:=\{(n,\omega): \omega B_n\subset S \}\subset \mathbb N\times \mathbb F_X$ is also computably enumerable.
By virtue of Proposition \ref{prima}, the set $\Omega_n:=\{\omega: (n,\omega)\in \Omega \}$ is non-empty for all $n\in\mathbb N$.
Let us denote by $\omega_n$ the  first element of $\Omega_n$ in the computable  enumeration of $\Omega$. Clearly $\{\omega_n \}_{n\in \mathbb N}$ is computably enumerable. Finally,
 $\omega_n B_n\cap \ker\pi \subset S\cap \ker\pi=\emptyset$ implies that $|\omega_n|_\Gamma>n$.
 Conversely,  $S:=\bigcup_{n\in\mathbb N} \omega_n B_n$ is computably enumerable, UB-generic and such that $S\cap \ker\pi=\emptyset$.\\
$(3)\iff(4)$\\
Suppose  that $\{\omega_n \}_{n\in \mathbb N}$ is a computably enumerable sequence of  $\mathbb F_X$  such that $|\omega_n|_\Gamma>n$ for all $n\in \mathbb N$.
We define inductively a subsequence $\{\omega_{k_n}\}_{n\in \mathbb N}$ as follows: $k_1:=1$ and $k_{n+1}:=|\omega_{k_n}|$ for all $n\geq 1$.
The new sequence is still computably enumerable and with the property that, for every $n\in\mathbb N$,
    $$k_{n}<\lvert \omega_{k_n}\rvert_\Gamma \leq k_{n+1},$$
and thus $|\omega_{k_{n+1}}|_\Gamma>|\omega_{k_n}|_\Gamma$ for all $n\in \mathbb N$. For the reverse implication notice that if the sequence of non-negative integers $\{|\omega_n |_\Gamma\}_{n\in \mathbb N}$  is  strictly increasing, then $\{|\omega_{n+2}|_\Gamma \}_{n\in \mathbb N}$  is superlinear. Therefore  $\{\omega_{n+2} \}_{n\in \mathbb N}\subset \mathbb F _X$ is a sequence with the desired property.\qedhere
  \end{proof}

The equivalence of (1) and (3) (and (4)), yields the following.
\begin{corollary}\label{coco}
If $\Gamma$ has solvable  UB-generic WP with respect to a finite generating set $X$ then it has solvable  UB-generic WP with respect to any finite generating set.
\end{corollary}
This shows that using the Upper Banach density to measure the Word Problem is a natural choice. In order to study  intrinsic properties of $\Gamma$ it is maybe even more natural  than the classical density.  In particular, the relation with the length in $\Gamma$ allows us to state the following corollary.
\begin{corollary}\label{pra}
Let  $\Gamma$  be a infinite, finitely generated group. If $\Gamma$ is algorithmically finite then it has unsolvable  UB-generic Word Problem and therefore unsolvable generic Word Problem with respect to any finite set of generators.
\end{corollary}
\begin{proof}
The existence of the sequence in (4) of Theorem \ref{principale} contradicts the definition of algorithmic finiteness.
\end{proof}  
Let us denote with $\mathcal C_1, \mathcal C_2$ and,$\mathcal C_3$ the classes of infinite finitely generated groups defined by the following properties:
\begin{align*}
&\mathcal C_1:=\{\Gamma: \mbox{ WP on } S\implies |\pi(S)|<\infty\},\\
&\mathcal C_2:=\{\Gamma: \mbox{ WP on } S^{-1}S   \implies |\pi(S)|<\infty\},\\
&\mathcal C_3:=\{\Gamma: \mbox{ WP on } S\implies S \mbox{ is not UB-generic}\}.
\end{align*}
Equivalently,
\begin{align*}
&\mathcal C_1=\{\Gamma: S \mbox{ computably enumerable}, |\pi(S)|=\infty \implies S\cap \ker \pi\neq \emptyset \},\\
&\mathcal C_2=\{\Gamma: S \mbox{ computably enumerable}, |S|=\infty \implies S^{-1}S\cap \ker \pi\neq e \},\\
&\mathcal C_3=\{\Gamma: S \mbox{ computably enumerable, UB-generic }  \implies S\cap \ker \pi\neq \emptyset \}.
\end{align*}
Note that the properties defining the classes $\mathcal C_1,\mathcal C_2$ and $\mathcal C_3,$ hold equivalently for one or for all finite sets of generators, which we omit in the definition.
We let  
 $$\mathcal C_4:=\{\Gamma: \mbox{ WP on  $S \subset \mathbb F_X$}    \implies S \mbox{ is not generic, for all finite generating sets } X\},$$
equivalently, 
$$\mathcal C_4=\{\Gamma: S\subset \mathbb F_X  \mbox{ c.\ e., generic }  \implies S\cap \ker \pi\neq \emptyset \mbox{ for all finite generating sets } X\}.$$
Finally let us denote with $\mathcal C_5$ the class of infinite finitely generated groups with unsolvable word problem.

The below chain of inclusions  swiftly follows from our investigation 
\begin{equation}\label{inclusions}
  \mathcal C_1\subset \mathcal C_2\subset \mathcal C_3\subset \mathcal C_4\subset \mathcal C_5.
 \end{equation} Since we can always assume that a finitely generated
 group  $\Gamma$ has solvable WP on a finite set $S$,   the WP  is
 solvable on its conjugacy closure~ $\overline{S}$. This easily
 implies that all the conjugacy classes of a  group in $\mathcal C_1$ are finite. 
It was already observed in~\cite{Myas/11} that algorithmically  finite
groups  (our class~$\mathcal C_2$) must be periodic. This is also the
case for groups  with unsolvable UB-generic WP (our class~$\mathcal C_3$): in fact, an element of infinite order provides a sequence like in (4) of Theorem~\ref{principale}. 
The computably presented groups in $\mathcal C_5$ have infinitely many
conjugacy classes,  as stated in \cite[Corollary~1]{SM}. Indeed,
assuming that  $\ker \pi$ is computably enumerable and that there
exists a finite  set $S\subset \mathbb F_X$ containing an element of
each non-trivial conjugacy class, it is easy to see that  $\mathbb F_X \setminus\ker \pi=\overline S\cdot \ker \pi$ and then also the set of non-trivial words is computably enumerable.

To our best knowledge, the only strict inclusion in \eqref{inclusions} is that of $\mathcal C_4$ in $\mathcal C_5$. Moreover, the only examples of computably presented groups in $\mathcal C _4$ actually live in $\mathcal C_2$. This inspires the following questions.
\begin{question}\label{que:strict}
  \emph{Are any of the inclusions in \eqref{inclusions} strict and/or trivial?}
Or one could  ask the same question within various subclasses of computably presented groups, such as
\begin{itemize}
\item residually finite (see also \cite[Problem~3.3]{Myas/11} and \cite{khou/14,Kly/15});
\item amenable (see also \cite[Problem~3.4]{Myas/11} and  \cite[Theorem~2.3]{gilman/10});   with computable F\o lner sets \cite{Cavaleri17}; of intermediate growth;  
\item sofic; with subrecursive sofic dimension (see \cite{Cavaleritesi}).
  \end{itemize}
\end{question}
If $\Gamma$ is amenable there exist notions of Banach densities with respect to a F\o lner sequence (see, for instance, \cite{lupini}), which are generalizations of the classical Banach density for~$\mathbb Z$, cf.~\cite{sumset}. Once more, these densities are closely related to those considered in~\cite{soleckihaar} where, moreover, it is proved that they exhibit peculiar behavior on amenable groups; cf.\ also \cite{banakh}. These notions  allow to formulate  the  extension of Erd\H os Sumset conjecture to amenable groups: if $A\subset \Gamma$ has positive UB-density (with respect to a F\o lner sequence) then there exist two infinite subsets $B, C\subset \Gamma$ such that $BC\subset A$ (this conjecture was recently proved in~\cite{prooferdos}). 
\begin{question}
\emph{Are the Banach densities of   $A$ in $\Gamma$ and $\pi^{-1}(A)$ in $\mathbb F_X$ related?} \end{question}
Analogous investigations were carried out in \cite{visible} for free abelian groups. 
A positive  answer to this question could lead to intriguing implications and new questions in relation with Question~\ref{que:strict} and the aforementioned conjecture--now theorem.


\end{document}